\documentclass[10pt]{amsart}
\title{Character series and Sklyanin algebras at points of order 2}
\author{Kevin De Laet}
\address{Department of Mathematics - Computer Sciences, University of Antwerp\\
Campus Middelheim - M.G. 223, Middelheimlaan 1, B-2020 Antwerpen (Belgium) \\ {\tt kevin.delaet2@uantwerpen.be}}
\date{}
\usepackage{amsmath,amsfonts,array,amscd,amsthm,amssymb,stackrel,verbatim,wrapfig,subfig,float}
\usepackage{enumerate,filecontents}
\usepackage{url}
\usepackage{tikz}
\usepackage[all]{xy}
\usetikzlibrary{arrows,matrix}
\usetikzlibrary{shapes.geometric}
\usetikzlibrary{decorations.markings} 

\tikzset{
  vertice/.style={circle,draw=black},
  decoration={markings,mark=at position 0.5 with {\arrow{>}}}
}

\newcommand{\wis}[1]{{\text{\em \usefont{OT1}{cmtt}{m}{n} #1}}}
\newcommand{\C}{\mathbb{C}}
\newcommand{\N}{\mathbb{N}}
\newcommand{\Z}{\mathbb{Z}}
\newcommand{\A}{\mathbb{A}}
\newcommand{\PP}{\mathbb{P}}

\makeatletter

\newcommand{\Rmnum}[1]{\expandafter\@slowromancap\romannumeral #1@}
\makeatother
\theoremstyle{plain}
\newtheorem{theorem}{Theorem}[section]
\newtheorem{lemma}[theorem]{Lemma}
\newtheorem{proposition}[theorem]{Proposition}
\newtheorem{corollary}[theorem]{Corollary}
\newtheorem{remark}[theorem]{Remark}
\newtheorem{example}[theorem]{Example}
\theoremstyle{definition}
\newtheorem{mydef}[theorem]{Definition}

\DeclareMathOperator{\Hom}{Hom}

\DeclareMathOperator{\rank}{rank}
\DeclareMathOperator{\Frac}{Frac}
\DeclareMathOperator{\Grass}{Grass}
\DeclareMathOperator{\Emb}{Emb}

\numberwithin{equation}{section}
\begin{document}
\sloppy

\begin{abstract}
This paper has two goals: to prove certain properties of character series of graded algebras on which a finite group acts as algebra automorphisms and to provide a detailed analysis of representations of 5-dimensional Sklyanin algebras at points of order 2. We also prove that for any odd prime $p$, the $p$-dimensional Sklyanin algebras associated to points of order 2 are graded Clifford algebras.
\end{abstract}
\maketitle
\tableofcontents
\section{Introduction}
In \cite{odesskiifeigin} Odesskii and Feigin constructed for each $n \geq 3$ the $n$-dimensional elliptic Sklyanin algebras, which are noncommutative graded algebras depending on an elliptic curve $E$ and a point $\tau\in E$ with $n$ generators and $\binom{n}{2}$ relations. By definition, these algebras are deformations of the polynomial ring in $n$ variables and the Heisenberg group of order $n^3$ acts on these algebras as gradation preserving automorphisms. In particular for $n=3$, these objects form the generic family of Artin-Schelter (AS) regular algebras. The AS regular algebras form a class of noncommutative graded algebras with excellent homological properties and correspond to noncommutative projective spaces. The AS regular algebras of global dimension 3 were classified in \cite{artin1987graded} by Artin and Schelter. Afterwards, they were intensively studied by Artin, Tate, Van den Bergh and others in for example \cite{artin2007some} and \cite{artin1991modules}. However, in higher dimensions, these objects remain a bit of a mystery and a classification is not in sight.

\par In \cite{DeLaet} the author showed that for each odd prime $p$ there is a $\frac{p-1}{2}$-dimensional family of graded Clifford algebras of global dimension $p$ on which $H_p$, the finite Heisenberg group of order $p$, acts as gradation preserving automorphisms. For $p=3$, these Clifford algebras correspond to 3-dimensional Sklyanin algebras associated to points of order 2. For $p \geq 5$, not all of these Clifford algebras are Sklyanin algebras. However, it was conjectured in \cite{DeLaet} that the $p$-dimensional Sklyanin algebras associated to points of order 2 are indeed Clifford algebras.
\par In \cite{DeLaet}, a notion of a \textit{character series} was defined for affine graded algebras. For a connected, positively graded algebra $A$ on which a reductive group $G$ acts as gradation preserving automorphisms, a character series decodes how $A$ decomposes as a $G$-module.
\begin{theorem}
Let $G$ be a finite group and $V$ a finite dimensional $G$-representation. Let $Z$ be an irreducible variety parametrizing $G$-deformations up to degree $k$ of an algebra $A$ and assume that 
$$
\forall x,y \in Z: H_{A_x}(t) = H_{A_y}(t).
$$
Then we have that $Z$ parametrizes $G$-deformations of $A$.
\label{th:characterseries}
\end{theorem}
This theorem is then applied for $G = H_p$ the finite Heisenberg group of order $p^3$ to calculate character series of homogeneous coordinate rings of elliptic curves embedded in $\PP^{p-1}$.
\par We then consider Sklyanin algebras associated to points of order 2. We prove in Section \ref{sec:Skly} the following theorem.
\begin{theorem}
The $p$-dimensional Sklyanin algebras with $p$ an odd prime associated to points of order 2 are graded Clifford algebras.\label{th:SklClif}
\end{theorem}
Using this theorem, we study the simple representations of the 5-dimensional Sklyanin algebras associated to points of order 2. The associated $\wis{Proj}(A)$ is determined by using this theorem and the statements from \cite{LeBruynCentral}.
\subsection{Acknowledgement}
I would like to thank Theo Raedschelders for the reading of the original form of this paper and his suggestions to make it more readable. In addition, I would like to thank professor Lieven Le Bruyn for his tremendous patience with me.
\subsection{Notation} In this article, we use the following notations:
\begin{itemize}
\item $\mathbf{V}(I)$ for $I \subset \C[a_1,\ldots,a_n]$ an ideal is the Zariski-closed subset of $\A^n$ or $\PP^{n-1}$ determined by $I$, it will be clear from the context if the projective or affine variety is used. 
\item $D(I)$ for $I$ an ideal $I \subset \C[a_1,\ldots,a_n]$ is the open subset $\A^n \setminus \mathbf{V}(I)$ or $\PP^{n-1} \setminus \mathbf{V}(I)$, it will be clear from the context if it is an open subset of affine space or of projective space. If $I = (a)$, then we write $D(a)$ for $D(I)$.
\item $\Z_n = \Z/n\Z$ for $n \in \mathbb{N}$.
\item $\Grass(m,n)$ will be the projective variety parametrizing $m$-dimensional vector spaces in $\C^n$.
\item For an algebra $A$ and elements $x,y \in A$, $\{x,y\} = xy + yx$.
\item For $V$ a $n$-dimensional vector space, we set $T(V) = \oplus_{k=0}^\infty V^{\otimes k}$, the tensor algebra over $V$.
\item Every graded algebra $A$ will be positively graded, finitely generated over $\C$ and connected, that is $A_0 = \C$.
\item The group $\wis{SL}_m(p)$ (respectively $\wis{PSL}_m(p)$) is the special linear group (respectively projective special linear group) of degree $m$ over the finite field with $p$ elements.
\item If $E$ is an elliptic curve over $\C$ and $n \in \N$, then $E[n]$ is the group of $n$-torsion points of $E$,
$$
E[n] = \{P \in E| [n]P = \underbrace{P + P + \ldots + P}_{n \text{ times}}=0\}.
$$
This group is always isomorphic to $\Z_n \times \Z_n$.
\item For any vector space $V$, $\C[V] = T(V)/(wv-vw|w,v \in V)$.
\item If $A$ is a connected, finitely generated, positively graded algebra, then the Hilbert series is defined as $H_A(t) = \sum_{k=1}^\infty \dim A_k t^k$.
\item Take a reductive group $G$ and 2 finite dimensional representations $V,W$ of $G$. Then $\Emb_G(V,W)$ is the set of injective linear $G$-maps from $V$ to $W$.
\item If $A$ is an algebra, then $\wis{Max}(A)$ is the set of maximal ideals of $A$.
\end{itemize}
\section{The finite Heisenberg group}
\label{sec:Heis}
Let $p$ be any odd prime. In this section we discuss the definition and the representation theory of the finite Heisenberg groups of order $p^3$. For a more thorough study of these groups we refer to \cite{Grassberger}.
\begin{mydef}
\textit{The Heisenberg group of order} $p^3$ is the finite group given by the generators and relations
\begin{displaymath}
H_p = \langle e_1,e_2,z | e_1^p = e_2^p=z^p = 1,e_1e_2 = ze_2e_1, e_1z = z e_1,e_2z = z e_2\rangle 
\end{displaymath}
\end{mydef}
$H_p$ is a central extension of the group $\mathbb{Z}_p \times \mathbb{Z}_p$, that is, we have the exact sequence
\begin{eqnarray}
 \xymatrix{1 \ar[r]& \mathbb{Z}_p \ar[r] & H_p \ar[r]&\mathbb{Z}_p \times \mathbb{Z}_p \ar[r]&  1}.
 \label{eqn:centralext}
\end{eqnarray}
The center of $H_p$ is generated by $z$. All of the 1-dimensional simple representations of $H_p$ are induced by the characters of $\mathbb{Z}_p \times \mathbb{Z}_p$. The other simple representations are $p$-dimensional and are determined by a primitive $p$th root of unity. They are defined in the following way: choose a primitive $p$th root of unity $\omega$, then define an action of $H_p$ on the vector space $V_1 = \mathbb{C}x_0 + \ldots + \mathbb{C}x_{p-1}$ by the rule
\begin{eqnarray}
e_1 \cdot x_i = x_{i-1}, & e_2 \cdot x_i= \omega^i x_i, &z \cdot x_i = \omega x_i ,
\label{eqnarray:simplereps}
\end{eqnarray}
indices taken in $\Z_p$. Taking another primitive root $\omega^i, 2\leq i \leq p-1$ determines a non isomorphic simple representation $V_i$. This implies that there are $p^2$ 1-dimensional and $p-1$ $p$-dimensional irreducible representations.
\par There are $p^2+p-1$ conjugacy classes in $H_p$, 1 for each central element and the other $p^2-1$ classes contain a unique element of the form $e_1^a e_2^b$, $a,b \in \mathbb{Z}_p, (a,b) \neq (0,0)$.
\par The character of a simple $p$-dimensional representation $V_i$ is given by 
\begin{align*}
\chi(z^k) &= p \omega^{ik} \\
\chi(e_1^a e_2^b) &= 0, (a,b) \neq (0,0).
\end{align*}
Such a representation $V_i$ also defines an antisymmetric bilinear form on the $\mathbb{Z}_p$-vector space $\mathbb{Z}_p \times \mathbb{Z}_p$. Identifying $e_1$ and $e_2$ with their images in $\mathbb{Z}_p \times \mathbb{Z}_p$, we get this form by setting $\langle e_1,e_2 \rangle = \omega^i$ and extending it $\Z_p$-linearly to $\mathbb{Z}_p \times \mathbb{Z}_p$, thus
$$
\langle a e_1 + b e_2,c e_1 + d e_2 \rangle = (\omega^i)^{ad-bc}.
$$
Let $\mu_p$ be the $p$th roots of unity in $\C$. If we define a group morphism $\xymatrix{\langle z \rangle \ar[r]^-\phi& \mu_p}$ by $\phi(z) = \omega^i$ (written multiplicatively in $\mu_p$), we get the commutative diagram below.
\begin{displaymath}
\xymatrix{H_p \times H_p \ar[r] \ar[d]^-{[,]} &\mathbb{Z}_p \times \mathbb{Z}_p \ar[d]^-{\langle , \rangle} \\
\langle z \rangle \ar[r]^-\phi & \mu_p }
\end{displaymath}
Since every $p$-dimensional representation is determined by the image of $z$, every nontrivial antisymmetric bilinear form on $\mathbb{Z}_p \times \mathbb{Z}_p$ uniquely defines a simple representation of $H_p$. Conversely, every simple $p$-dimensional representation of $H_p$ defines a unique nontrivial antisymmetric bilinear form on $\mathbb{Z}_p \times \mathbb{Z}_p$ by extending linearly $\langle e_1,e_2 \rangle = \frac{\chi(z)}{p}$.
\par From now on, fix a simple $p$th root of unity $\omega$. We will write $V_i$ for the simple $p$-dimensional representation for which $\chi_{V_i}(z) = p \omega^i$ and $\chi_{a,b}$ for the 1-dimensional representation defined by \begin{eqnarray}
\chi_{a,b}(e_1) = \omega^a, &\chi_{a,b}(e_2) = \omega^b.
\end{eqnarray}

\begin{remark}
When $n\geq 3$ is not prime, $H_n$ is still a central extension of $\Z_n \times \Z_n$ with $\Z_n$ and the exact sequence from \ref{eqn:centralext} is still valid. The simple $n$-dimensional representations are also uniquely determined by a primitive $n$th root of unity and defined as in \ref{eqnarray:simplereps}, but there are other simple representations of dimension $1 < k < n$ to consider (cfr. \cite{Grassberger}).
\end{remark}
\subsection{Heisenberg geometry on $\PP^{p-1}$}
Let $V= V_1$ be the associated simple $H_p$-representation. Then this defines a group action of $\Z_p \times \Z_p$ on $\PP^{p-1} = \PP(V)$ by the composition
$$
\xymatrix{H_p \ar[r]^-\phi & \wis{GL}(V) \ar[r]^-p & \wis{PGL}(V)}.
$$
It is clear that the center of $H_p$ is the kernel of $\psi = p \circ \phi$. The group $\wis{SL}_2(p)$ acts on $\Z_p \times \Z_p$ in the obvious way as group automorphisms. This action of $\wis{SL}_2(p)$ can be extended to an action on $H_p$ in a way that the center of $H_p$ is invariant under this action. From this it follows that each $p$-dimensional simple representation $V_i$ we have
$$
\forall \psi \in\wis{SL}_2(p): V_i \cong V_i^{\psi} \text{ as $H_p$-representation},
$$
with $V_i^\psi$ the $H_p$-representation one gets by twisting the morphism $\xymatrix{H_p \ar[r]^-\phi & \wis{GL}(V)}$ with the automorphism $\psi$ of $H_p$. In particular, this implies that there exists a group homomorphism
$$
\xymatrix{\wis{SL}_2(p) \ar[r] & \wis{PGL}(V)}.
$$
Together with the induced action of $\Z_p \times \Z_p$ on $\PP(V)$, this defines a projective representation of $(\Z_p \times \Z_p) \rtimes \wis{SL}_2(p)$ on $\PP(V)$.
\begin{proposition}
There exists a bijection as $\wis{SL}_2(p)$-sets between the following sets
$$
\xymatrix{ \PP^1_{\Z_p} \ar@{<->}[r] & \{\Z_p \times \Z_p-\text{orbits with non trivial stabilizer} \}}
$$
\end{proposition}
\begin{proof}
See for example \cite[Lemma 5.2]{DeLaet}.
\end{proof}
\subsection{The Heisenberg group and elliptic curves}
\label{sub:Mod}
As is well known, the modular group $\Gamma = \wis{PSL}_2(\mathbb{Z})$ acts on the complex upper half-plane
\begin{displaymath}
\mathbb{H}=\{x+iy | y>0\}
\end{displaymath}
by M\"obius transformations. The fundamental domain of this action defines isomorphism classes of elliptic curves and its compactification, made by adding the $\Gamma$-orbit $\overline{\mathbb{Q}}=\mathbb{Q}\cup \{\infty\}$, is the Riemann sphere $S^2$. In general, one can take any other group of finite index $G$ in $\Gamma$, find its fundamental domain in $\mathbb{H}$ and check what information a point in this domain holds. The modular curve $X'(p)$, $p$ prime is made this way by taking $G = \Gamma(p)$, with
\begin{displaymath}
\Gamma(p) = \left\lbrace \begin{bmatrix}
a & b \\ c & d
\end{bmatrix} \in \Gamma | a,d \equiv 1 \bmod p , b,c \equiv 0 \bmod p \right\rbrace.
\end{displaymath}
In order to compactify $X'(p)$, one needs to add cusps to get $X(p)$.
A point of $X'(p)$ holds 3 pieces of information:
\begin{itemize}
\item an elliptic curve $(E,O)$,
\item an embedding of $\mathbb{Z}_p \times \mathbb{Z}_p$ into $E$ or equivalently, two generators $e_1,e_2$ of $E[p]$.
\item a primitive $p$th root of unity $\omega$ such that $\langle e_1,e_2 \rangle = \omega$, where this bilinear antisymmetric form is found by the Weil-pairing.
\end{itemize}
$X(p)$ has an action of $\wis{PSL}_2(p) = \Gamma/\Gamma(p)$ by definition. This action is defined by taking another set of generators of $E[p]$, $f_1,f_2$, but their inner product must still remain $\omega$. This defines a $\wis{SL}_2(p)$-action, but since $-I_2$ works trivially on $X(p)$, we have a $\wis{PSL}_2(p)$-action.
\par We know that a bilinear antisymmetric form on $\mathbb{Z}_p \times \mathbb{Z}_p$ defines a unique simple $H_p$-representation $V$. Let $P_1,P_2$ be a generating set of $E[p]$ and denote $P_{a,b} = [a]P_1+[b]P_2$, then there exists a (unique up to multiplication with a scalar) function $f$ on $E$ with divisor
\begin{displaymath}
-(P_{0,0} + \ldots + P_{0,p-1})+ P_{p-1,0}+\ldots + P_{p-1,p-1}.
\end{displaymath}
In \cite{Silverman} it is proved that there exists a primitive $p$th root of unity $\omega$ such that $\omega = \frac{f}{\phi_{P_2}^*(f)}$, where $\phi_{P}^*$ stands for the pullback under the morphism
\begin{displaymath}
\xymatrix{E \ar[r]^-\phi & E}, \xymatrix{\tau \ar@{|->}[r] &\tau + P}.
\end{displaymath}
Calculating the divisor, one finds that the function 
\begin{displaymath}
N(f)=f \phi_{P_1}^*(f)\ldots (\phi_{P_1}^*)^{p-1}(f)
\end{displaymath}
is constant and not 0, which means we can rescale $f$ so that $N(f) = 1$. We will now define an action of $H_p$ on the vector space
\begin{displaymath}
H^0(E,\mathcal{O}(P_{0,0} + \ldots + P_{0,p-1})).
\end{displaymath}
Let $x_0 = 1$ and define
\begin{eqnarray}
e_1 \cdot g = f\phi_{P_1}^*(g),& e_2 \cdot g= \phi_{P_2}^*(g).
\end{eqnarray}
If we set $x_i =e_1^{p-i}\cdot x_0$, we find that 
\begin{eqnarray}
e_1\cdot x_i = x_{i-1}, & e_2\cdot x_i = \omega^i x_i.
\end{eqnarray}
This defines our action of $H_p$. These global sections define an embedding of $E$ into $\mathbb{P}^{p-1}$ and it is clear that the defining equations will be $H_p$-invariant.
\section{Graded Clifford algebras}
\label{sec:GradCliff}
We will work with graded Clifford algebras. This section will deal with the particular case we are interested in, but this is not the general definition. For more information, see \cite{LeBruynCentral}.
\begin{mydef} Let $R=\C[y_1,\ldots,y_n]$ be a polynomial ring in $R$ variables, graded such that $\deg y_i = 2, 1 \leq i \leq n$  and let $M$ be a symmetric matrix with entries in $R_2$, $\det(M) \neq 0$. Then \textit{the graded Clifford algebra} $A(M)$ associated to $M$ is the algebra generated by $x_1,\ldots,x_n,y_1\ldots,y_n$ with relations
\begin{eqnarray}
x_i x_j + x_j x_i = M_{ij},& [x_i,y_j]= 0,& [y_i,y_j]=0, 1 \leq i,j\leq n
\end{eqnarray}
and $\deg(x_i) = 1, 1 \leq i \leq n$.
\end{mydef}
\begin{proposition}
$A(M)$ is a free module of rank $2^n$ over $\C[y_1,\ldots,y_n]$.
\end{proposition}
The center of $A(M)$ depends on the parity of $n$:
\begin{itemize}
\item If $n$ is even, then $Z(A(M)) = \C[y_1,\ldots,y_n]$, a polynomial ring in $n$ variables.
\item If $n$ is odd, then $Z(A(M)) = \C[y_1,\ldots,y_n,g]$, with $g$ a central element of degree $n$ fulfilling the relation $g^2 = \det (M)$.
\end{itemize}
\subsection{Representations of graded Clifford algebras}
In order to describe the representation theory of $A(M)$, we repeat the definition of the PI degree of an algebra finite over it's center.

\begin{mydef}
Let $A$ be a finite module over it's center $Z(A)$, with $Z(A)$ a normal domain. Then we define the \textit{PI degree} of $A$ to be $\sqrt{\dim_K A \otimes_{Z(A)} K}$, with $K =\Frac Z(A)$.
\end{mydef}
The PI degree of an algebra $A$ finite over it's center is equal to
$$a=\max\{m \in \N| \exists \phi: A \ \wis{M}_m(\C) \text{ simple}\}.$$
The set of points $\mathfrak{m} \in \wis{Max}(Z(A))$ such that there exists a maximal ideal $\mathfrak{M} \in \wis{Max}(A)$ such that $\mathfrak{M} \cap Z(A) = \mathfrak{m}$ and $A/\mathfrak{M} \cong \wis{M}_a(\C)$ is called the Azumaya locus of $A$.
\par Returning to the the special case of graded Clifford algebras, let $\mathfrak{m}$ be a maximal ideal of $\C[y_1,\ldots,y_n]$ and let $n$ be odd. It is easy to see (using the theory of Clifford algebras over $\C$, see for example \cite{Clifford2000} and \cite{LeBruynCentral}) that the dimension of simple representations depends on the rank of $M$ after taking the quotient with respect to $\mathfrak{m}$. Let $Y_\mathfrak{m}$ be the corresponding symmetric matrix in $\wis{M}_n(\C)$.
\begin{itemize}
\item If $\rank Y_\mathfrak{m} = n$, then there are 2 simple $2^{\frac{n-1}{2}}$-dimensional representations. These 2 representations are separated in the center by the cover $\xymatrix{\wis{Max}(Z(A)) \ar@{->>}[r] & \wis{Max}(\C[y_1,\ldots,y_n])}$ coming from the inclusion $\xymatrix{\C[y_1,\ldots,y_n] \ar@{^{(}->}[r] & Z(A)}$.
\item If $\rank Y_\mathfrak{m} = n-1$, then there is 1 simple $2^{\frac{n-1}{2}}$-dimensional representation lying over $\mathfrak{m}$.
\item If $\rank Y_\mathfrak{m} = k$, $k$ odd, then there are 2 $2^{\frac{k-1}{2}}$-dimensional simple representations lying over $\mathfrak{m}$.
\item If $\rank Y_\mathfrak{m} = k$, $k$ even, then there is 1 $2^{\frac{k}{2}}$-dimensional simple representation lying over $\mathfrak{m}$.
\end{itemize}
\subsection{The $\wis{Proj}$ of graded Clifford algebras}
In noncommutative algebraic geometry, one studies $\wis{Proj}(A)$, which is the quotient category of all graded $A$-modules by the full subcategory of graded torsion modules. Of particular interest are the linear modules, that is, left graded critical $A$-modules with Hilbert series $\frac{1}{(1-t)^n}$ for some $n$. This $n$ is called the dimension of the module. If $n=1$, one speaks of point modules, $n=2$ are line modules, and so on.
\par However, in some cases there are other modules to consider: fat point modules. These fat points are critical modules with in their class in $\wis{Proj}(A)$ a representative module with Hilbert series $\frac{e}{(1-t)}$ with $e > 1$. This $e$ is called the multiplicity of the corresponding module. In the spirit of noncommutative algebraic geometry, these fat points and point modules correspond to the simple objects in $\wis{Proj}(A)$.
\par For a graded Clifford algebra $A(M)$ with $M$ the corresponding symmetric matrix in $\wis{M}_n(\C[y_1,\ldots,y_n])$, the classification of the fat points and point modules is determined by \cite[Proposition 9]{LeBruynCentral} on the condition that $A(M)$ is AS regular. For a maximal graded prime ideal $p$ of $\C[y_1,\ldots,y_n]$, let $M(p)$ be the symmetric matrix one gets after specialization with respect to $p$.
\begin{proposition}
\label{prop:projcliffod}
Let $Y = \wis{Proj}(\C[y_1,\ldots,y_n]) = \PP^{n-1}$ and let
$$Y_k = \{p \in Y | \rank M(p) = k\} \subset Y.$$
Then for each $p \in Y$, there exists a unique graded prime ideal $P$ of $A(M)$ such that $A(p)=A(M)/P \otimes_{\C[y]} \C[y,y^{-1}]$ has the following structure:
\begin{itemize}
\item If $p \in Y_k$ with $k$ odd, then we have the isomorphism
$$
A(p) \cong \wis{M}_{2^{\frac{k-1}{2}}}(\C[x,x^{-1}])(\underbrace{0,0,\ldots,0}_{2^{\frac{k-1}{2}}})
$$
and $\deg(x) = 1$. This implies that there is 1 fat point of multiplicity $2^{\frac{k-1}{2}}$.
\item If $p \in Y_k$ with $k$ even, the isomorphism becomes
$$
A(p) \cong \wis{M}_{2^{\frac{k}{2}}}(\C[y,y^{-1}])(\underbrace{0,0,\ldots,0}_{2^{\frac{k}{2}-1}},\underbrace{1,1,\ldots,1}_{2^{\frac{k}{2}-1}})
$$
and $\deg(y)=2$. This implies that there are 2 fat points of multiplicity $2^{\frac{k}{2}-1}$.
\end{itemize}
\end{proposition}
It is easy to see now that the points $p \in \PP^{n-1}$ for which there are point modules in $\wis{Proj}(A)$ are determined by $Y_2$. If we let $X_k = \bigcup_{i=0}^k Y_i$, then this proposition shows that the point modules of $\wis{Proj}(A)$ determine a 2-to-1 cover of $X_2$, with ramification over $X_1$.
\par We will now show an example to determine the representations of 3-dimensional Sklyanin algebras and points of order 2.
\begin{example}
The 3-dimensional Sklyanin algebras associated to points of order 2 correspond to quotients $A_t$ of the algebra $\C \langle x,y,z\rangle$ by the relations
$$
\begin{cases}
yz+zy = t x^2,\\
zx+xz = t y^2,\\
xy+yx = t z^2,
\end{cases}
$$
for $t\in \C \setminus \{0,2,2\omega,2\omega^2,-1,-\omega,-\omega^2\}$, $\omega$ a primitive 3rd root of unity. The associated symmetric matrix over $\C[x^2,y^2,z^2]$ is
$$
M = \begin{bmatrix}
2x^2 & tz^2 & ty^2\\ tz^2 & 2y^2 & tx^2 \\ ty^2 & t x^2 & 2 z^2
\end{bmatrix}.
$$
It follows that the center of $A_t$ is equal to $\C[x^2,y^2,z^2,g]$ and 1 relation of the form $g^2 = \det(M)$. The equation $\det(M)$ determines an elliptic curve $E$ in $\PP^2_{[x^2,y^2,z^2]}$. The point modules can be found by putting $g = 0$, from which it follows that there are 2 point modules lying over the elliptic curve $\wis{Proj}(\C[x^2,y^2,z^2]/(\det(M))$.
\end{example}
\begin{remark}
These fat points and point modules correspond to $\C^* \times \wis{PGL}_n(\C)$-orbits in $\wis{rep}^{ss}_n(A)$, cfr. \cite{DeLaetLeBruyn}. Here for a graded algebra $A$ which is a finite module over its center, we define
$$
\wis{rep}^{ss}_n	(A) = \{\xymatrix{A \ar[r]^-\phi & \wis{M}_n(\C)}| \exists z \in Z(A)_k, k > 0: \phi(z) \neq 0 \},
$$
see \cite{BocklandtSymens} and \cite{DeLaetLeBruyn} for more information.
\end{remark}
\section{Shioda's modular surface $S(p)$}
\label{sec:Shioda}
We will need to know the construction and exceptional fibers of Shioda's elliptic modular surface. Although this surface can be constructed for each $n \geq 3$, we will only consider the case $n = p$ prime.
\begin{mydef}
\textit{Shioda's modular surface} is an elliptic surface over the modular curve $X(p)$. Let $\omega$ be any primitive $p$th root of unity. The fibers over the points of $X'(p)$ (the modular curve minus the cusps) parametrize elliptic curves with level $p$ structure, that is, elliptic curves $(E,O)$ and embeddings
$$\xymatrix{\Z_p \times \Z_p \ar@{^{(}->}[r] & E}
$$ such that the following diagram commutes
\begin{displaymath}
\xymatrix{\Z_p \times \Z_p \ar[r] \ar[dr]^-{[,]} & E[p] \ar[d]^-{\langle , \rangle} \\
  & \Z_p \cong \mu_p }
\end{displaymath}
The bracket $[-,-]$ is an antisymmetric $\Z_p$-bilinear form defined by $[(1,0),(0,1)] = \omega$ and $\langle -,- \rangle$ is the Weil pairing on $E[p]$ (cfr. \cite{Silverman}, III.8).
\end{mydef}
The exceptional fibers of $S(p)$ lie above the cusps of $X(p)$. Each exceptional fibre is the union of $p$ lines, each line intersecting exactly 2 other lines.
\par $S(p) \subset \PP^{p-1} \times X(p)$ and therefore comes with 2 natural projection maps, $\pi_1$ and $\pi_2$. For each fiber $\pi_2^{-1}(x), x \in X(p)$, the projection map $\pi_1$ is injective. From the construction of $S(p)$ in \cite{barth1985projective}, it follows that the Heisenberg group acts on each fiber (with the center acting trivially) and that $\wis{SL}_2(p)$ acts on the set of fibers. More importantly, the projection map 
$$\xymatrix{S(p) \ar[r]^-{\pi_1} &\PP^{p-1}}
$$
is a map of $\Z_p \times \Z_p \rtimes \wis{SL}_2(p)$-sets.
\begin{proposition}
Any cycle of $p$ lines that forms an $H_p$-orbit of lines can be sent to any other such cycle using the $\wis{SL}_2(p)$-action. 
\end{proposition}
\begin{proof}
Using the $\wis{SL}_p(2)$-action, we can assume that our cycle of $p$ lines goes through the point $y_0 = (1:0:0:\ldots:0)$, as the $H_p$-orbit of $y_0$ is the same as the fixed points of $e_2$. Then there exists a unique $1\leq k \leq \frac{p-1}{2}$ such that this cycle contains the line through $y_0$ and $e_1^k(y_0)$. Take now the element
$$N=
\begin{bmatrix}
k^{-1} &0 \\ 0 & k
\end{bmatrix} \in \wis{SL}_2(p).
$$
Then under the new action of $\Z_p \times \Z_p$ defined by $n$, we have $y'_1 = (e'_1)^{-1} y_0 = e_1^k y_0$ and so our cycle has become the cycle of $p$ lines with one of the lines through $y_0$ and $y'_1$.
\end{proof}
There are a total of $(p+1)\frac{p-1}{2}$ such cycles of $p$ lines in $\PP^{p-1}$, which is equal to the number of cusps of $X(p)$ as expected.
\par  The center of $\wis{SL}_2(p)$ acts on $\PP^{p-1}$ by the involution $\phi$ determined by $x_i \leftrightarrow x_{-i}$, indices taken $\bmod p$, cfr. \cite{fisher2009pfaffian}. This involution determines the inverse map on the elliptic curves embedded in $\PP^{p-1}$.
\begin{proposition}
The points of order 2 on an elliptic curve $E$ correspond to eigenspaces of eigenvalue 1 of $\phi$.
\end{proposition}
\begin{proof}
According to \cite[Proposition 3.7]{fisher2009pfaffian}, the point $O \in E$ corresponds to an eigenvector with eigenvalue $-1$ of $\phi$. $O$ lies on the hyperplane $\mathbf{V}(x_0)$ and the other $p-1$ points of $E$ that lie on this hyperplane are points of order $p$. As a hyperplane can intersect $E$ in maximal $p$ points, it follows that the first coordinate of any point of order 2 on $E$ is not 0. As such a point still has to be fixed by $\phi$, it follows that it's eigenvalue is 1.
\end{proof}
\subsection{The case $p=5$}
The following subsection is a summary of Chapter \Rmnum{4}, Section 5 of \cite{hulek}. This gives a concrete example of Shioda's elliptic modular surface for $p=5$. Recall the representation $V_1$ of $H_5$ of Section \ref{sec:Heis} and let $\PP^4=\PP(V_1)$ be the associated projective space on which $\Z_5 \times \Z_5$ acts by way of the composition
$$
\xymatrix{H_5 \ar[r] & \wis{GL}(V) \ar[r] & \wis{PGL}(V)}.
$$
\begin{theorem}
Every elliptic curve $(E,O)$ can be embedded in $\PP^4$ in such a way that $O$ is sent to $O_a=(0:1:a:-a:-1)$ for some $a \in \C$ and such that the action of $\Z_5 \times \Z_5$ on $\PP^4$ is an extension of the action of $E[5]$ on $E$. Call this embedded curve $C_a \subset \PP^4$ with $O=O_a$. The relations of $C_a$ are given by $ax_i^2 + a^2 x_{i+1}x_{i-1} - x_{i+2}x_{i-2}$, $ 0 \leq i \leq 4$, indices taken in $\Z_5$. 
\end{theorem}
Let $X'(5)$ be the modular curve parametrizing elliptic curves with level $5$ structure. Let $X(5)$ be the natural compactification of $X'(5)$. Then the surface $S(5) \subset \PP^4_{[x_0:\ldots:x_4]} \times \PP^1_{[A:B]}$ defined by the relations $AB x_i^2 + A^2 x_{i+1}x_{i-1} - B^2 x_{i+2}x_{i-2}$, $ 0 \leq i \leq 4$, indices taken in $\Z_5$ is Shioda's elliptic modular surface $S(5)$.
\par  $X(5) \cong \PP^1$ and the map between $S(5)$ and $\PP^1$ is determined by the projection $\pi_2$ on the second factor. $\wis{PSL}_2(5)$ acts on $\PP^1$ such that points belong to the same orbit if and only if the corresponding fibers of $\pi_2$ in $S(5)$ are isomorphic as varieties. For every point $p \in X(5)$ the fiber $C_p = \pi_2^{-1}(p)$ is an elliptic curve, except for the $\wis{PSL}_2(5)$-orbit of $0$. For the $\wis{PSL}_2(5)$-orbit of 0, the fiber is a cycle of 5 lines, each line intersecting exactly 2 other lines.
\begin{theorem}
The projection $\pi_1$ of $S(5)$ to $\PP^4$ is a determinantal surface $S_{15}$, defined by taking the $3 \times 3$-minors of the matrix
$$
\begin{bmatrix}
x_0^2   & x_1^2   & x_2^2   & x_3^2   & x_4^2 \\
x_2 x_3 & x_3 x_4 & x_4 x_0 & x_0 x_1 & x_1 x_2 \\
x_1 x_4 & x_2 x_0 & x_3 x_1 & x_4 x_2 & x_0 x_3 \\
\end{bmatrix}.
$$
This projection is $1$-to-$1$ except for the 30 points of $\PP^4$ with a non-trivial stabilizer in $H_5$, for these points $\pi_1$ is $2$-to-$1$. These 30 points are the singular points of $S_{15}$ and the projection map $\xymatrix{S(5) \ar[r] & S_{15}}$ is a desingularization of these points.
\end{theorem}
Let $\phi$ be the involution on $\PP^4$ defined by $x_i \leftrightarrow x_{-i}$. We will need the points of order 2 of each elliptic curve $E$, that is, the intersection of $S_{15}$ with the plane of $\PP^4$ that corresponds to the 3-dimensional subspace of $V_1$ that is invariant $\phi$.
\begin{proposition}
The intersection of $S_{15}$ with the plane containing the 2-torsion points of $C_a$ for all $a$ such that $C_a$ is an elliptic curve is determined by the equations
$$
\begin{cases}
x_1-x_4=0,\\
x_2-x_3=0,\\
x_0^4x_1x_2-x_0^2x_1^2x_2^2-x_0(x_1^5+x_2^5)+2x_1^3x_2^3=0.
\end{cases}
$$
\end{proposition}
\label{prop:order2}
\begin{proof}
See amongst others \cite{hulek}. Alternatively, one can check by computer that the intersection of the plane $\mathbf{V}(x_1-x_4,x_2-x_3)$ with $S_{15}$ is indeed the claimed curve.
\end{proof}
\section{Constructing $G$-algebras}
\label{sec:Galg}
In order to talk about character series, one first needs to know the construction of $G$-algebras. This is a generalization of the setting considered in \cite{DeLaet}.
\begin{mydef}
Let $G$ be a reductive group. We call a positively graded connected algebra $A$, finitely generated in degree 1, a \textit{$G$-algebra} if $G$ acts on it by gradation preserving automorphisms.
\end{mydef}
This implies that there exists a representation $V$ of $G$ such that $T(V)/I \cong A$ with $I$ a graded ideal of $T(V)$, which is itself a $G$-subrepresentation of $T(V)$.
\par The general construction of quadratic $G$-algebras with relations is as follows. Let $V$ be a $G$-representation. Then $V \otimes V$ is also a $G$-representation which decomposes as a summation of simple representations, say $V \otimes V \cong \oplus_{i=1}^m S_i^{a_i}$ where the $S_i$ are distinct simple representations of $G$ and $a_i \geq 0$. A $G$-algebra $A$ is then constructed by taking embeddings of the $S_i$ in $V \otimes V$ as relations of $A$.
\par One can of course do the same for other degrees by taking relations in $T(V)_i = V^{\otimes i} \cong  \oplus_{j=1}^m S_j^{a_j}$ and take different embeddings of the simple representations of $G$ in $V^{\otimes i}$ as relations.
\begin{mydef}
Let $A$ be a $G$-algebra with corresponding ideal $I$ of $T(V)$. We call $B$ a \textit{$G$-deformation} of $A$ up to degree $k$ if $B$ is also a quotient of $T(V)$ such that $\forall 1 \leq i \leq k: A_i \cong B_i$ as $G$-representations. We will call $B$ a $G$-deformation if $\forall i \in \N: A_i \cong B_i$ as $G$-representations.
\end{mydef}
If the relations for $A$ are all of the same degree $k$, then all $G$-deformations up to degree $k$ of $A$ depend on a product of Grassmannians. For example, let $A$ be a quadratic algebra of which we want to find all $G$-deformations up to degree 2. Let $I_2 = \oplus_{i=1}^m S_i^{e_i} \subset V \otimes V =\oplus_{i=1}^m S_i^{a_i}$, then the $G$-deformations up to degree 2 are parametrized by the product $\Emb_G(\oplus_{i=1}^m S_i^{e_i},\oplus_{i=1}^m S_i^{a_i})=\prod_{i=1}^m \Grass(e_i,a_i)$.
\par In general, the total set of $G$-deformations up to degree $k$ of a $G$-algebra $A = T(V)/I$ are determined by a Zariski closed subset of 
$$
Z_k=\prod_{i=1}^k \prod_{S \text{ simple}}\Grass(e_{i},a_{i})
$$
where $I_i = \oplus_{S \text{ simple}} S^{e_i} \subset T(V)_i = \oplus_{S \text{ simple}} S^{a_i}$
\begin{mydef}
We say that a variety $Z$ \textit{parametrizes $G$-deformations up to degree $k$} of a $G$-algebra $A$ if $\xymatrix{Z \ar@{^{(}->}[r]^-\phi & Z_k}$ can be embedded in $Z_k$ and the point corresponding to $A$ in $Z_k$ belongs to the image of $\phi$. We say that $Z$ \textit{parametrizes $G$-deformations} of $A$ if $Z$ parametrizes $G$-deformations up to degree $k$ for some $k$ and for each point $x \in Z$ with corresponding algebra $A_x$, we have
$$
\forall i \in \N: (A_x)_i \cong A_i \text{ as $G$-representations}.
$$
\end{mydef}
\par We will now show 2 examples of $G$-deformations.
\begin{example}
Let $G = H_3$ and $V_1 = \C x_0 \oplus \C x_1 \oplus \C x_2$ with the action of $H_3$ defined by $e_1 \cdot x_i = x_{i-1}$ and $e_2 \cdot x_i = \omega^i x_i$, $\omega$ a primitive 3rd root of unity. It is easy to see that $V_1^* = V_2$. Then $V_1 \otimes V_1 \cong (V_2)^{\oplus 3}$, decomposed as
\begin{eqnarray}
H_3\cdot(x_1 x_2 + x_2 x_1) \oplus H_3\cdot (x_1 x_2 - x_2 x_1) \oplus  H_3 \cdot (x_0^2)
\end{eqnarray}

The subrepresentation $V_1 \wedge V_1$ is generated by $x_1 x_2 - x_2 x_1$ over $H_3$. The ideal generated by this representation is of course the relations one needs to get the polynomial ring $\C[V_1]$.
\par In order to find $H_3$-deformations up to degree 2 of $\C[x_0,x_1,x_2]$, we need to find the $H_3$-embeddings of $V_2$ into $(V_2)^{\oplus 3}$, which is determined by a vector by Schur's lemma. Such an embedding is completely determined by an element $$A(x_1 x_2 + x_2 x_1)+B(x_1 x_2 - x_2 x_1) + C x_0^2$$ with $(A:B:C) \in \Grass(1,3) \cong \PP^2$. Putting $a=A+B, b = A-B, c=C$, one gets the familiar relations for the 3-dimensional Sklyanin algebras
$$\begin{cases}
a x_1 x_2 + b x_2 x_1 + c x_0^2,\\
a x_2 x_0 + b x_0 x_2 + c x_1^2,\\
a x_0 x_1 + b x_1 x_0 + c x_2^2.
\end{cases}
$$
In fact, we will see in Theorem \ref{th:Sklychar} that, whenever $[a:b:c]$ defines a Sklyanin algebra, this algebra is an $H_3$-deformation of $\C[V_1]$.
\end{example}
\begin{example}
Apart from the 3-dimensional Sklyanin algebras, one also has the twisted coordinate ring $\mathcal{O}_\tau(E)$, which is a quotient of the Sklyanin algebra $\mathcal{A}_\tau(E)$ by a central element of degree 3. One can easily check that this element is fixed by the Heisenberg action. Therefore, the twisted coordinate rings are $H_3$-deformations of the graded coordinate ring $\mathcal{O}(E)$, where $E$ is embedded in $\PP^2$ in Hesse normal form.
\end{example}
\subsection{Character series}
\label{sec:kos}
Given a $G$-algebra $A$, it is a natural question to ask how $A$ behaves as a $G$-module. As $G$ acts as gradation preserving automorphisms, we have a decomposition
$$
A = \bigoplus_{k=0}^{\infty} \bigoplus_{S \text{ simple}} S_{e_{k,S}}
$$
with almost all $e_{k,S}$ equal to 0. We will only consider the case that $G$ is finite.
\begin{mydef}
Let $G$ be a finite group. The \textit{character series} for an element $g \in G$ and for a $G$-algebra $A$ is a formal sum 
\begin{displaymath}
Ch_A(g,t) = \sum_{n \in \mathbb{Z}} \chi_{A_n}(g) t^n.
\end{displaymath}
\end{mydef}
For example, if $g =  1$, $Ch_A(1,t) = H_A(t)$, the Hilbert series of $A$. As a character of a representation is constant on conjugacy classes, we can represent the decomposition of $A$ in simple $G$-representations as a vector of length equal to the number of conjugacy classes and on the $i$th place the character series $Ch_A(g,t)$ with $g \in C_i$, the $i$th conjugacy class.
\begin{lemma}
Let $V$ be a simple representation of $G$ and let $A$ be a $G$-algebra constructed from $T(V)$. For every element $z$ of the center, we have that $Ch_A(z,t) = H_A(\lambda t)$, where $z$ acts on $V$ by multiplication with $\lambda$.
\label{lem:cent}
\end{lemma}
\begin{proof}
It follows that in degree $k$ the action of $z$ on $A_k$ is given by multiplication with $\lambda^k$, so the character series for the element $z$ in this case is given by $$Ch_A(z,t)=\sum_{k=0}^\infty \lambda^k \dim A_k t^k = H_A(\lambda t).$$ 
\end{proof}
\par If $A$ is in addition a Koszul algebra, there is a nice duality between the character series of $A$ and the character series of it's Koszul dual $A^!$, discovered in \cite{ZhangJing}. We recall the definition of a Koszul algebra and it's quadratic dual.
\begin{mydef}
A connected, positively graded quadratic algebra $A$ is Koszul if the trivial module $\C$ has a free linear resolution.
\end{mydef}
\begin{mydef}
Given a quadratic algebra $A = T(V)/(R)$ with generators $V = \mathbb{C}x_0+\ldots+\mathbb{C}x_n$, we define the \textit{quadratic dual} $A^!$ to be the quadratic algebra $T(V^*)/(J_2)$, with $J_2$ defined as the subspace of $V^* \otimes V^*$ such that $\forall w \in J_2, \forall v \in R_2: w(v) = 0$. If $A$ is Koszul, then we call $A^!$ the Koszul dual of $A$.
\end{mydef}
Some standard properties of Koszul algebras we will need are that there is a relation between the Hilbert series of $A$ and $A^!$, given by
\begin{displaymath}
H_A(t)H_{A^!}(-t) = 1
\end{displaymath}
and that $A$ is Koszul iff $A^!$ is Koszul.
\par Because the Koszul complex is a free resolution of the trivial module $\mathbb{C}$, which is isomorphic as $G$-representation to the trivial representation and because the Koszul complex consists of $G$-morphisms, we have a similar formula for finding the character series of the Koszul dual as we have for the Hilbert series. More precisely, we have
\begin{align}
Ch_A(g,t) Ch_{(A^!)^*}(g,-t) = 1.
\label{al:chKos}
\end{align}
This allows us to compute $Ch_{A^!}(g,t)$ whenever we know $Ch_A(g,t)$. To know the character series of $A^!$, we have to take the complex conjugates of the coefficients of $Ch_{(A^!)^*}(g,t)$. In short, for a Koszul algebra $A$, the character series associated to $A$ is completely determined by the character series of $A^!$.
\subsection{Application to polynomial rings}
Let $V_1$ be the $p$-dimensional simple representation of $H_p$ for which the center acts on by multiplication with $\omega$. For the polynomial ring $\C[V_1]$, one has the advantage that the Koszul dual $\C[V_1]^! = \wedge V_{p-1}^*$ is a finite dimensional algebra and therefore easier to decompose in simple $G$-representations. In \cite{DeLaet} the author calculated the character series of $\C[V_1]$ using this technique. There it was shown that
\begin{align*}
Ch_{\C[V_1]}(1,t)&=\frac{1}{(1-t)^p},\\
Ch_{\C[V_1]}(z^k,t)&=\frac{1}{(1-\omega^k t)^p},\\
Ch_{\C[V_1]}(e_1^ke_2^l,t)&=\frac{1}{1-t^p}, (k,l)\neq (0,0).
\end{align*}
In general, if one has a $G$-algebra $A$ of finite global dimension which is Koszul, then it is easier to calculate the character series of $A^!$. This in turn then determines the character series of $A$.
\section{Character series are constant}
\label{sec:Charconst}
We will now show that under good conditions, character series are indeed constant for $G$-deformations.
\begin{lemma}
Let $A$ be a $G$-algebra with $\xymatrix{ T(V) \ar@{->>}[r]^-{p} & A}$ the natural projection map. Decompose $A_k = \oplus_{i=1}^m S_i^{\oplus e_i}$ into simple $G$-representations and similarly $T(V)_k = \oplus_{i=1}^m S_i^{\oplus a_i}$ with naturally $a_i \geq e_i$. Then there exists a subspace $W \subset T(V)_k$ such that $W \cong A_k$ as $G$-representations and $p|_W$ is an isomorphism of $G$-representations.
\end{lemma}
\begin{proof}
It follows from Schur's lemma that the map $\xymatrix{ T(V)_k \ar@{->>}[r]^-{p_k} & A_k}$ is a surjective element of $\oplus_{i=1}^m\Hom(S_i^{\oplus a_i},S_i^{\oplus e_i})\cong\oplus_{i=1}^m \Hom(\C^{\oplus a_i},\C^{\oplus e_i})$. There it reduces to a statement of linear maps, which follows from standard linear algebra.
\end{proof}
More importantly, this means that if $A_k = \oplus_{i=1}^m S_i^{\oplus e_i}$, we can choose $G$-generators $v_{i,j} \in A_k,1 \leq i \leq m, 0\leq j \leq e_i$ and elements $w_{i,j},1 \leq i \leq m, 0\leq j \leq e_i$ of $T(V)_k$ such that $G \cdot w_{i,j} \cong S_i \cong G \cdot v_{i,j}$ and $p(w_{i,j}) = v_{i,j}$.
\par Now we can prove Theorem \ref{th:characterseries}.
\begin{proof}[Proof of Theorem \ref{th:characterseries}]
For $x \in Z$, let $\xymatrix{ T(V) \ar@{->>}[r]^-{p_x} & A_x}$ denote the natural homomorphism. Suppose that the character series of $A_x$ and $A_y$ are not the same. There exists a minimal $l \geq k$ such that $(A_x)_l \ncong (A_y)_l$ as $G$-modules. According to the lemma we can find a subspace $W$ of $T(V)_l$ such that $p_x(W) = (A_x)_l$ and $W \cong (A_x)_l$ as $G$-representations. Then there exists an open subset $U \subset Z$ with $x \in U$ such that the images of the chosen generators of $W$ are linearly independent $\forall z \in U$. This implies that $W \cap \ker(p_z)_l = 0$. As we know that all algebras have the same Hilbert series and $W$ is a $G$-representation, this automatically implies that $W \cong (A_z)_l$ as $G$-representations.
\par Similarly there exists an open subset $U'$ with $y \in U'$ and a subspace $W'$ of $T(V)_l$ such that $W' \cong (A_y)_l$ and $((p_z)_l)|_{W'}$ a $G$-isomorphism $\forall z \in U'$. As $Z$ was irreducible, there exists a point $a \in U \cap U'$. But then it follows that
$$
(A_x)_l \cong W \cong (A_a)_l \cong W' \cong (A_y)_l
$$
as $G$-representations, which is a contradiction.
\end{proof}
\begin{corollary}
Let $Z$ be a connected variety such that each irreducible component of $Z$ fulfils the conditions of Theorem \ref{th:characterseries}.  Then the character series $Ch_{A_x}(g,t)$ are also constant on $Z$ $\forall g \in G$.
\label{cor:conn}
\end{corollary}
\subsection{Application to elliptic curves}
Let $$S = \mathbf{V}(x_{i+k}x_{-i+k},1 \leq i \leq \frac{p-3}{2}, 0\leq k \leq p-1)$$ be a cycle of $p$ lines in $\PP^{p-1} = \PP(V_{p-1})$ and let $$\mathcal{O}(S)=\C[V_1]/(x_{i+k}x_{-i+k},1 \leq i \leq \frac{p-3}{2}, 0\leq k \leq p-1)$$ be the associated graded ring.
It is easy to see that these relations form an $H_p$-orbit in $\C[V_1]_2$.
\begin{theorem}
The character series of $\mathcal{O}(S)$ is given by
\begin{gather*}
Ch(1,t) = \frac{1+(p-2)t+t^2}{(1-t)^2}\\
Ch(z^k,t)=\frac{1+(p-2)\omega^k t+(\omega^k t)^2}{(1-\omega^k t)^2}, 1\leq k \leq p-1\\
Ch(e_1^a e_2^b z^k,t)=1 \text{ if } (a,b)\neq (0,0),0\leq k \leq p-1
\end{gather*}
\end{theorem}
\begin{proof}
All this follows if we can construct an easy basis of $\mathcal{O}(S)_k$ for each $k \geq 2$. In degree 2, a basis is given by $x_i^2,x_i x_{i+1}, 0 \leq i \leq p-1$. Using this, we can make in degree $k$ the basis 
\begin{gather*}
x_i^l x_{i+1}^{k-l}, 1 \leq l \leq k-1, 0\leq i \leq p-1\\
x_i^k \text{ if } 0\leq i \leq p-1
\end{gather*}
From this it follows that the Hilbert series is given by $1 + \sum_{j=1}^\infty jp t^j = \frac{1+(p-2)t+t^2}{(1-t)^2}$. The character series for elements of the center of $H_p$ is then correct by Lemma \ref{lem:cent}. For the other elements of $H_p$, we see that the action of $e_1$ is given by a permutation on this basis without any fixed elements and for $e_2$, the action on $x_i^k$ is given by multiplication by $\omega^{ik}$ and on $x_i^l x_{i+1}^{k-l}$ by $\omega^{il}\omega^{(i+1)(k-l)}=\omega^{ik-l}$. If $k \not\equiv 0 \bmod p$, then the character $\chi_{\mathcal{O}(S)_k}(e_2)$ is 0 as we then have $\mathcal{O}(S)_k \cong V^{\oplus k}$ with $V$ a simple $p$-dimensional representation of $H_p$. If $k \equiv 0 \bmod p$ and $k \neq 0$, we get the summation of the $p$th roots of unity $\frac{k}{p}$ times, which is 0. As the characters of $e_1,e_2$ and $z$ determine the character series for $H_p$, we are done.
\end{proof}
From Theorem \ref{th:characterseries} it follows that
\begin{corollary}
The character series of the coordinate ring of an elliptic curve embedded in $\PP^{p-1}$ with $p\geq 5$ prime such that $H_p$ acts on it by translation with $p$-torsion points has as character series
\begin{gather*}
Ch(1,t) = \frac{1+(p-2)t+t^2}{(1-t)^2},\\
Ch(z^k,t)=\frac{1+(p-2)\omega^k t+(\omega^k t)^2}{(1-\omega^k t)^2},1\leq k \leq p-1\\
Ch(e_1^a e_2^b z^k,t)=1 \text{ if } (a,b)\neq (0,0), 0\leq k \leq p-1.
\end{gather*}
\end{corollary}
\begin{proof}
In Subsection \ref{sub:Mod}, we have seen that any elliptic curve $E$ can be embedded in $\PP^{p-1}$ with $p$ prime such that the finite Heisenberg group acts on $\PP^{p-1}$ by a projectivication of one of its irreducible $p$-dimensional representations. $E$ is stable under this action and the action of $H_p$ on $E$ reduces to translation with $p$-torsion points.
\par In particular, this means that the relations of $E$ in $\C[x_0,\ldots,x_{p-1}]$ form an $H_p$-stable subspace and that the graded coordinate ring $\mathcal{O}(E)$ is an $H_p$-algebra. For $p\geq 5$ the relations of $O(E)$ are quadratic.
\par In Section \ref{sec:Shioda} we have seen that these families of elliptic curves are parametrized by $X'(p)$. This family degenerates to a cycle of $p$ lines over the cusps in $X(p)= \overline{X'(p)}$ and one of these cycles is determined by the equations $x_{i+k}x_{-i+k},1 \leq i \leq \frac{p-3}{2}, 0\leq k \leq p-1$. The Hilbert series is constant for each point on $X(p)$.
\par Now apply Corollary \ref{cor:conn} with $Z= X(p)$ and $G = H_p$.
\end{proof}
\section{Sklyanin algebras}
\label{sec:Skly}
In \cite{odesskiifeigin} Odesskii and Feigin constructed Sklyanin elliptic algebras using $\theta$-functions on lattices of every dimension $n$. By construction, each Sklyanin algebra is a $H_n$-algebra. Tate and Van den Bergh showed in \cite{VdBTate} showed that these Sklyanin algebras have the same Hilbert series as the polynomial ring in $n$ variables. Let $Q_{n,1}(E,\tau)$ be a Sklyanin algebra associated to the elliptic curve $E$ and a point $\tau \in E$.
\begin{theorem}
Let $n=p$ be prime. Then the character series associated to $H_p$ of the Sklyanin algebras is the same as the character series of the polynomial ring in $p$ variables.
\label{th:Sklychar}
\end{theorem}
\begin{proof}
Let $E$ be an elliptic curve with $\tau \in E$ and let $V_1$ be the simple representation of $H_p$ associated to $\omega$ a primitive $p$th root of unity. As $Q_{p,1}(E,\tau) = T(V_1)/I_\tau$ is an $H_p$-algebra and the relations of $Q_{p,1}(E,\tau)$ are quadratic, this implies that the $\frac{p(p-1)}{2}$-relations form an $H_p$-stable subspace of $V_1 \otimes V_1 \cong V_2^{\oplus p}$. In particular, this implies a morphism
$$
\xymatrix{ E \ar[r] & \Grass(\frac{p-1}{2},p)} \cong \Emb_{H_p}(V_2^{\oplus\frac{p-1}{2}},V_2^{\oplus p})
$$
whose image contains the embedding determined by $V_1 \wedge V_1$, which are the relations of the polynomial ring $\C[V_1]$. Then the theorem follows from Theorem \ref{th:characterseries} by putting $Z = E$ and $G= H_p$.
\end{proof}
In \cite{DeLaet}, the author calculated for regular Clifford algebras which are $H_p$-deformations of the polynomial ring $\C[V_1]$ the character series using the fact that such Clifford algebras are free modules over a polynomial ring. It was remarked that these algebras have the same character series as $\C[V_1]$. However, we can now give a new proof using Corollary \ref{cor:conn}.
We first need to prove Theorem \ref{th:SklClif}.
\begin{proof}[Proof of Theorem \ref{th:SklClif}]
According to Odeskii and Feigin (cfr. \cite{odesskiifeigin}), the relations for the $n$-dimensional Sklyanin algebra associated to the elliptic curve $E$ and the point $\tau \in E$ are given by 
$$
\sum_{r \in \Z_n} \frac{\theta_{j-i}(0)}{\theta_{j-i-r}(-\tau)\theta_{r}(\tau)} x_{j-r}x_{i+r}=0	
$$
with $i,j \in \Z_n, i \neq j$ and $\theta_{0}(z),\ldots,\theta_{n-1}(z)$ $\theta$-functions of order $n$. Let $n=p$ be an odd prime. We know that the defining relations of a Sklyanin algebra form an $H_p$-stable subspace of $V_1 \otimes V_1 \cong V_2^{\oplus p}$, with $V_2$ the simple $p$-dimensional representation associated to $\omega^2$. It is therefore sufficient to consider a basechange for $i = -j$ to see if we get relations of the form $x_i x_{-i} + x_{-i}x_i= a_i x_0^2$. As $\tau = -\tau$ as $\tau$ is of order 2, we find that the relations are given by $H_p$-orbits of
$$
\sum_{r \in \Z_p} \frac{\theta_{2j}(0)}{\theta_{2j-r}(\tau)\theta_{r}(\tau)} x_{j-r}x_{-j+r}=0	
$$
with $1\leq j \leq \frac{p-1}{2}$. For fixed $j$ and $r$ the coefficient of $x_{-j+r}x_{j-r} = x_{j-(2j-r)}x_{-j+(2j-r)}$ is
$$
\frac{\theta_{2j}(0)}{\theta_{2j-(2j-r)}(\tau)\theta_{2j-r}(\tau)}=\frac{\theta_{2j}(0)}{\theta_{r}(\tau)\theta_{2j-r}(\tau)}
$$
which is exactly the same as the coefficient of $x_{j-r}x_{-j+r}$. This means that every relation belongs to the vector space generated by the $\{x_j,x_{-j}\} = x_j x_{-j} + x_{-j} x_j, 1\leq j \leq \frac{p-1}{2}$, and $x_0^2$. This vector space is $\frac{p+1}{2}$-dimensional in which we need to find a $\frac{p-1}{2}$-dimensional subspace, call this space $R$. Let $A$ be the $\frac{p-1}{2}\times \frac{p-1}{2}$-matrix with on place $(i,j)$ the coefficient from  $\{x_j,x_{-j}\}$ coming from the $i$th equation and let $b$ be the column vector with on the $i$th place the coefficient of $x_0^2$ of the $i$th equation. Let $M$ be the matrix $[A,b]$, then the rank of this matrix is $\frac{p-1}{2}$. There are now 2 options:
\begin{enumerate}[(a)]
\item The rank of $A$ is $\frac{p-1}{2}$: then we are done as this means that up to basechange in $R$ we have relations of the form $\{x_i,x_{-i}\} + a_ix_0^2 = 0$.
\item The rank of $A$ is $\frac{p-3}{2}$: this means that there is a column in $A$ which is a linear combination of the other ones. As the rank of $M$ is $\frac{p-1}{2}$, this necessarily means that $\{x_i,x_{-i}\} = a_i \{x_k,x_{-k}\}$ and $x_0^2= a_0 \{x_k,x_{-k}\}$ for some $1\leq k  \leq \frac{p-1}{2}$. As we have that $x_0^2 \neq 0$, this implies that $a_0 \neq 0$. But this means that $A$ should have rank $\frac{p-1}{2}$, a contradiction.
\end{enumerate}
We have proved that there exists a basis of the relations of the form $\{x_i,x_{-1}\} + a_ix_0^2 = 0$, which are indeed graded Clifford algebras.
\end{proof}
\begin{remark}
This theorem is most certainly false if $n$ is even: in this case, there can be a coefficient before $x_{\frac{n}{2}}^2$ which is not 0. Therefore, the proof does not work. Also, there are too many elements in the center for $Q_{n,1}(E,\tau)$ to be a Clifford algebra when $n$ is even (cfr. \cite{odesskiifeigin}).
\end{remark}
\begin{corollary}
Let $p$ be an odd prime. The graded regular Clifford algebras with relations defined by $H_p$-orbits of $\{x_i, x_{-i}\} = a_i x_0^2, 1\leq i \leq \frac{p-1}{2}$ have the same character series as the polynomial ring in $p$ variables.
\end{corollary}
\begin{proof}
The moduli space of graded regular $H_p$-Clifford algebras is an open subset $U$ of $\PP^{\frac{p-1}{2}}$ which intersects the moduli space of Sklyanin algebras in an infinite number of points. As we already know that the character series of all the Sklyanin algebras are the same and $U$ is irreducible, we know from corollary \ref{cor:conn} that the character series are the same.
\end{proof}
We now adopt the notation of \cite{DeLaet} for $H_p$-Clifford algebras.
\begin{mydef}
The algebra $C(a_0:\ldots:a_{\frac{p-1}{2}})$  with $(a_0:\ldots:a_{\frac{p-1}{2}})\in \PP^{\frac{p-1}{2}}$ stands for the $H_p$-algebra with generators $x_0,\ldots,x_{p-1}$ and $H_p$-representatives as relations
$$
\begin{cases}
a_0\{x_{i},x_{-i}\} = a_i x_0^2, 1\leq i\leq \frac{p-1}{2},\\
a_{i+1}\{x_{i}x_{-i}, x_{-i}x_{i}\} = a_i\{x_{(i+1)}x_{-(i+1)}, x_{-(i+1)}x_{(i+1)}\}, 1 \leq i \leq \frac{p-3}{2}.
\end{cases}
$$
\end{mydef}

Let $p$ be any odd prime. In \cite{DeLaet} it was shown that the $H_p$-Clifford algebras form a $\frac{p-1}{2}$-dimensional family of $\Grass(\frac{p-1}{2},p)$. On the other hand, the $p$-dimensional Sklyanin algebras form a 2-dimensional family. Then Theorem \ref{th:SklClif} implies that in $\Grass(\frac{p-1}{2},p)$ the moduli space of Sklyanin algebras and the moduli space of $H_p$-Clifford algebras intersect in a curve.
\section{Sklyanin algebras of global dimension 5 and points of order 2}
\label{sec:Sklypoint2}
When the global dimension of the Sklyanin algebras is 5, we can find the equations for the Sklyanin algebras $Q_{5,1}(E,\tau)$ with $\tau \in E$ of order $2$. We know that the relations can be written as
\begin{align*}
\{x_{1+k},x_{4+k}\} = a x_k^2, 0 \leq k \leq 4, \\
\{x_{2+k},x_{3+k}\} = b x_k^2, 0 \leq k \leq 4,
\end{align*}
with associated quadratic form
$$Q=
\begin{bmatrix}
2x_0^2  & bx_3^2  & a x_1^2 & a x_4^2 & b x_2^2 \\
b x_3^2 & 2x_1^2  & b x_4^2 & a x_2^2 & a x_0^2 \\
a x_1^2 & b x_4^2 & 2 x_2^2 & b x_0^2 & a x_3^2 \\
a x_4^2 & a x_2^2 & b x_0^2 & 2 x_3^2 & b x_1^2 \\
b x_2^2 & a x_0^2 & a x_3^2 & b x_1^2 & 2 x_4^2
\end{bmatrix}
$$
over the polynomial ring $\C[x_0^2,x_1^2,x_2^2,x_3^2,x_4^2]$.
\begin{theorem}
If the Clifford algebra $C(1:a:b)$ determines a Sklyanin algebra $Q_{5,1}(E,\tau)$ with $\tau$ of order $2$, then the point $(1:a:b)$ lies on the affine curve $$C'=\mathbf{V}(-a^3b^3+a^5+b^5+2a^2b^2-8ab) \subset \A^2_{(a,b)}.$$ 
The elliptic curve $E'=E/\langle \tau \rangle$ is given by the curve $C_t$ as defined in Section \ref{sec:Shioda} with 
$$
t = \frac{a^3b-b^3-2a^2}{a^4-ab^2-4b}.
$$
\end{theorem}
\begin{proof}
For a graded Clifford algebra $C$ with associated quadratic form $Q$ over a polynomial ring $R$, the point modules are determined by the graded prime ideals where $Q$ has rank $ \leq 2$ after specialisation (that is, the rank of the Clifford algebra $C/P$ over $R/P$ is $\leq 2$ where $P$ is a non-trivial maximal graded prime ideal of $R$), see for example \cite{LeBruynCentral}. For a $n$-dimensional Sklyanin algebra with $n \neq 4$, the variety parametrizing point modules is equal to the elliptic curve $E$, as proved by Smith in \cite{smithpoint}. The corresponding variety in $\wis{Proj}(R)$ is $E'=E/\langle \tau \rangle$.
\par In any case, this means that one has to find $t \in \C$ such that the $3 \times 3$-minors are all 0 on a curve of the form $C_t$. In particular, all the $3 \times 3$-minors in the point $(0:1:t:-t:-1)$ should be 0. Using Macaulay2, one sees that at least two equations of all the $3 \times 3$-minors are given by $-2b^2t^3+2ab^2t-2a^2$ and $-a^2bt^2-ab^2t+2a^2$. Then one can eliminate the variable $t$ to get the following equation
$$
a(-a^3b^3+a^5+b^5+2a^2b^2-8ab)=0.
$$
The projective closure of the line $\mathbf{V}(a)$ is not closed under the action of $\wis{PSL}_2(5)$ on $\PP^2$ which was constructed in \cite{DeLaet}, so we necessarily see that the curve $C'$ parametrizing 5-dimensional Sklyanin algebras associated to points of order 2 is given by the equation $$-a^3b^3+a^5+b^5+2a^2b^2-8ab=0.$$ From this one deduces that
$$
t = -\frac{(a^3-2b^2)a}{-a^3b^2+b^4+2a^2b}.
$$
However, in the fraction field of $C'$, $t$ is equal to $\frac{a^3b-b^3-2a^2}{a^4-ab^2-4b}$. As $t$ is not constant, neither are $a$ and $b$.
\end{proof}
Some remarks about $C'$:
\begin{itemize}
\item Let $C$ be the projective closure of $C'$ in $\PP^2_{[A:B:C]}$. In \cite{DeLaet} the author constructed a $\wis{PSL}_2(5)$-action on $\PP^2_{[A:B:C]}$ such that points lying in the same orbit gave isomorphic algebras. The curve $C$ is stable under this action.
\item $C$ has 6 singularities: the $\wis{PSL}_2(5)$-orbit of the point $[1:0:0]$, whose points give the algebras isomorphic to the quantum space with relations $x_i x_j + x_j x_i = 0, i\neq j$. One of course expects these 6 points to be singular if one looks at the degeneration of the point modules: for each 5-dimensional Sklyanin algebra, the point modules are given by the elliptic curve $E$. Such a family of elliptic curves degenerates to a cycle of 5 lines, so if $C$ was smooth, one expects the quantum space to have as point modules 1 cycle of 5 lines. However, a quantum $\PP^n$ has at least in its point scheme the full graph on $n+1$ points.
\item There are 12 points on $C$, which are smooth points, but the corresponding algebras are all isomorphic to the algebra with relations $x_i^2 = 0, x_i x_{i+1} + x_{i+1}x_i=0, 0 \leq i \leq 4$ (that is, they lie in the same $\wis{PSL}_2(5)$-orbit as the point $[0:0:1]$). These 12 points form the intersection of $C$ with the curve $\mathbf{V}(AB+C^2)$, which parametrizes the Koszul dual of the graded coordinate rings of all elliptic curves with level-5 structure. However, these 12 points do not give the Koszul dual of graded coordinate rings of elliptic curves, but of a cycle of 5 lines.
\end{itemize} 
\subsection{Simple representations} Before we start the description of representations of 5-dimensional Sklyanin algebras associated to points of order 2, we prove a lemma regarding the possible 1-dimensional representations of $H_p$-Clifford algebras. 
\begin{lemma}
The only $H_p$-Clifford algebras with non-trivial 1-dimensional representations are the algebras isomorphic to the quantum algebra $\C_{-1}[x_0,\ldots,x_{p-1}]$.
\end{lemma}
\begin{proof}
The algebras isomorphic to the quantum algebra are given by the $\wis{PSL}_2(p)$-orbit of the point $(1:0:\ldots:0)$, which we know consists of $p+1$ elements (see \cite[Theorem 5.1]{DeLaet}). Apart from this point, the other points in this orbit are given by the action of the element $\begin{bmatrix}
1 & 0 \\ 1 & 1
\end{bmatrix}$ on the point $(1:2:\ldots:2)$. These other points are elements of the open subset $D(a_1 \cdots a_{\frac{p-1}{2}})\subset \bigcup_{i=1}^{\frac{p-1}{2}} D(a_i)$ and there are exactly $p$ of them. We will prove that the number of points in $\bigcup_{i=1}^{\frac{p-1}{2}} D(a_i)$ for which there exists a 1-dimensional non-trivial representation is equal to $p$. Suppose that $a_i \neq 0$ and that there exists a non-trivial 1-dimensional representation. Using the Heisenberg action, we may assume that $x_0$ is not sent to $0$. Using the $\C^*$-action, we may also assume that the image of $x_0$ is $1$. Let $y_i \in \C$ be the image of $x_i$ in $\C$ under this representation. We then have
$$
2 y_i y_{-i} = a_i, 2y_{(k+1)i} y_{(k-1)i} = a_i y_{ki}^2
$$
As $a_i \neq 0$, we necessarily have $y_i \neq 0$.
\par We have the formula 
\begin{eqnarray}
y_{ki} = \frac{a_i^{\binom{k}{2}}y_i^k}{2^{\binom{k}{2}}},
\label{eq:1dim}
\end{eqnarray}
which is trivially true for $k = 0,1$. The other cases can be proved by induction.
\par In particular, for $k = p$ we get $y_0 = \frac{a_i^{\binom{p}{2}}y_i^p}{2^{\binom{p}{2}}} = 1$.
For $k = p+1$ we find $y_{i} = \frac{a_i^{\binom{p+1}{2}}y_i^{p+1}}{2^{\binom{p+1}{2}}}$. This implies
$$
y_i^p = \frac{2^{\binom{p}{2}}}{a_i^{\binom{p}{2}}}=\frac{2^{\binom{p+1}{2}}}{a_i^{\binom{p+1}{2}}}
$$
So we have $a_i^p = 2^p$. It follows that $y_i$ is a $p$th root of unity. From equation \ref{eq:1dim} follows that $y_{ki}$ is uniquely determined for all $k$, but as $i \neq 0$, this means that all $y_j$ are uniquely determined by $y_i$. But the $a_j$ are determined by $2y_j y_{-j} = a_j$. So the number of points in $\bigcup_{i=1}^{\frac{p-1}{2}} D(a_i)$ with non-trivial 1-dimensional representations is less than or equal to $p$. As we know that there are certainly $p$ points in this set, we are done.
\end{proof}
As $A=Q_{5,1}(E,\tau)$ is a graded Clifford algebra with quadratic form $Q$ over $\C[x_0^2,x_1^2,x_2^2,x_3^2,x_4^2]$, it follows that the center $Z(A)$ of $A$ is given by the 5 elements $x_i^2, 0 \leq i \leq 4$ of degree 2 and the square root of the determinant of $Q$, call this element $c_5$ of degree 5. These 6 elements satisfy one relation of degree 10
$$
\phi(x_0^2,x_1^2,x_2^2,x_3^2,x_4^2) = \det Q = c_5^2.
$$
with $\phi$ an $H_5$-invariant polynomial of degree 5 (this follows from \cite[Proposition 5.7]{DeLaet}). As $A$ is a free module of rank $2^5$ over the polynomial ring $\C[x_0^2,x_1^2,x_2^2,x_3^2,x_4^2]$, it follows that the PI-degree of $A$ is $2^{\frac{5-1}{2}}=2^2=4$, that is, generically the simple representations are 4-dimensional. The dimension of a simple representation lying above a point in $\wis{Max}(Z(A))$ is determined by the rank of the quadratic form $Q$ after specialization. Let $S =  \wis{Max}(Z(A))$, $\A^5=\wis{Max}(\C[x_0^2,x_1^2,x_2^2,x_3^2,x_4^2])$ and denote $\psi$ for the double cover $\xymatrix{S \ar@{->>}[r]^-\psi & \A^5}$ coming from the inclusion $\xymatrix{\C[x_0^2,x_1^2,x_2^2,x_3^2,x_4^2] \ar@{^{(}->}[r]& Z(A)}$. This cover is ramified over $\mathbb{V}(\det Q)$. We will use the degree of the central elements coming from its inclusion in $A$, so the $x_i^2$ have degree 2.
\begin{theorem}
Above any point of $S$ lies a 4-dimensional simple representation of $A$ with exception of the cone above the elliptic curve $E'$. For points on the cone above the elliptic curve $E'$ there exists a unique 2-dimensional simple representation, except for the trivial representation lying above the maximal ideal $(x_0^2,x_1^2,x_2^2,x_3^2,x_4^2,c_5)$.
\end{theorem}
\begin{proof}
The fact that on the cone above $E'$ there are only simple representations of dimension  $\leq 2$ follows from the fact these representations are representations of the twisted coordinate ring $\mathcal{O}_\tau(E)$, which is of PI-degree 2. Also, as none of the Sklyanin algebras are isomorphic to the quantum space, all these representations are necessarily of dimension 2. Let $I \subset \C[x_0^2,x_1^2,x_2^2,x_3^2,x_4^2]$ be the ideal associated to the cone above $E'$, which is generated by 5 elements of degree 4. Let $J_k$ be the ideal of $\C[x_0^2,x_1^2,x_2^2,x_3^2,x_4^2]$ generated by all the $k \times k$ minors of the quadratic form $Q$. Using Macaulay2, one checks that $(J_3)_6 = I_6$, that is, the degree 6 elements in  $J_3$ and $I$ are the same. For $J_4$, one checks that all the $4 \times 4$-minors of $Q$ are the generators for $I^2$.
\par For the open set $D(J_4)$ the rank of $Q$ is 4 or 5, in both cases the corresponding simple representation is 4-dimensional.
\end{proof}
One even has the following description of $\phi$.
\begin{proposition}
Let $Sec^2 E \subset \PP^4$ be the secant variety of an elliptic curve $E$ embedded in $\PP^4$, that is, the Zariski-closure of all lines in $\PP^4$ that intersect $E$ in 2 points. We have
$$\mathbf{V}(\phi) = Sec^2 E' \subset \PP^4_{[x_0^2:\ldots:x_1^2]} $$
\end{proposition}
\begin{proof}
According to Proposition \Rmnum{8}.2.5 of \cite{hulek}, if $E$ is an elliptic curve given as the intersection of 5 quadrics $Q_i=z_i^2+t z_{i+1} z_{i+4} -\frac{1}{t}z_{2+i}z_{i+3}, 0 \leq i \leq 4$ in $\PP^4$, then the defining equation for $Sec^2 E$ is given by
$$
\det\left(\frac{\partial Q_i}{\partial z_j}\right)=0.
$$
One checks by a computer computation that $\det\left(\frac{\partial Q_i}{\partial x_j^2}\right)=0$ in the quotient ring $\C[x_0^2,x_1^2,x_2^2,x_3^2,x_4^2]/(\det Q)$. As both $\det\left(\frac{\partial Q_i}{\partial x_j^2}\right)$ and $\det Q$ are of the same degree, this implies that they are equal to each other up to a scalar.
\end{proof}
The (cone over) the secant variety is singular along (the cone over) $E'$, which is equal to the ramification locus of $A$. This follows also from proposition 5 of \cite{LeBruynCentral}, as the codimension of the ramification locus is 3 and the global dimension of these algebras is finite.
\subsection{From $\wis{rep}_n(A)$ to $\wis{Proj}(A)$}
We have found that there are 3 different types of non-trivial simple representations:
\begin{itemize}
\item Representations of dimension 4 where $c_5$ does not act trivial.
\item Representations of dimension 4 where $c_5$ does act trivial.
\item Representations of dimension 2.
\end{itemize}
These representations are determined by the rank of the quadratic form $Q$. Let $Y = \PP^4_{[x_0^2,\ldots,x_4^2]}$. We set
$$
Y_k = \{ p \in Y | \rank Q(P) = k\}
$$
where $Q(P)$ is the matrix in $\wis{M}_5(\mathbb{C}[t])$ with $t$ of degree 2 one gets after taking the quotient of $Q$ by the graded ideal determined by $P$. We have found that
\begin{itemize}
\item $Y_5 = Y \setminus \mathbf{V}(\phi)$,
\item $Y_4 = \mathbf{V}(\phi) \setminus E'$,
\item $Y_2 = E'$.
\end{itemize}
\par Using these facts, we can determine the fat points and point modules for these algebras using Proposition \ref{prop:projcliffod}.

\begin{theorem} The fat points and point modules of $\wis{Proj}(A)$ are determined by:
\begin{itemize}
\item For each point on $\wis{Proj}(Z(A))\setminus \mathbf{V}(\phi)$, we have one corresponding fat point module of multiplicity 4 in $\wis{Proj}(A)$.
\item For each point on $\mathbf{V}(\phi) \setminus E'$, there are 2 corresponding fat point modules of multiplicity $2$.
\item For each point on $E'$, there are 2 point modules in $\wis{Proj}(A)$. The point modules of $A$ are given by $E$ and the map between $E$ and $E'$ is the natural isogeny $\xymatrix{E \ar[r] & E'}$.
\end{itemize}
\end{theorem}
\begin{proof}
All this follows from Proposition \ref{prop:projcliffod} applied to this special case.	
\end{proof}
In \cite{OdesskiiFeigincenter} Odesskii and Feigin proved that for an elliptic curve $E$ and $\tau$ a torsion point of order $m$, the center of the Sklyanin algebra $Q_{n,1}(E,\tau)$ for $n$ odd is generated by a central element $c_n$ of degree $n$ and $n$ algebraically independent elements of degree $m$, with one relation of the form 
$$
\phi(u_0,u_1,\ldots,u_{n-1}) = c_n^m.
$$
Moreover, $Q_{n,1}(E,\tau)$ is a finite module over its center. For $n=5$ and assuming $(3,m)=(5,m)=1$, these results with respect to points of order 2 might suggest the following regarding the $\wis{Proj}$ of these algebras:
\begin{enumerate}
\item The polynomial $\phi(u_0,u_1,u_2,u_3,u_4)$ is the relation corresponding to the secant variety $Sec^2(E')$, where $E' = E/\langle \tau \rangle$.
\item The PI-degree of $Q_{n,1}(E,\tau)$ is $m^2$.
\item The open set $D(c_5)\subset \wis{Proj}(\C[u_0,u_1,u_2,u_3,u_4,c_5])$ corresponds to fat points of multiplicity $m^2$. 
\item For each point in $\mathbf{V}(c_5) \setminus E'$, there are $m$ fat points of multiplicity $m$.
\item For a point on $E'$, there are $m$ point modules coming from the isogeny $\xymatrix{E \ar[r]& E'}$.
\end{enumerate}
The last part is an easy consequence from the fact that these point modules are point modules of the twisted coordinate ring $\mathcal{O}_{[3]\tau}(E)$. The other conjectures will be the subject of future work.

\end{document}